\documentclass[12pt,twoside,reqno]{amsart}
\usepackage{amsmath}
\usepackage{amsfonts}
\usepackage{amssymb}
\usepackage{color}
\usepackage{mathrsfs}
\usepackage{cite}
\usepackage{geometry}
\usepackage{marginnote}
\usepackage{todonotes}
\allowdisplaybreaks
\newtheorem{theorem}{Theorem}[section]

\newtheorem{lemma}[theorem]{Lemma}
\newtheorem{proposition}[theorem]{Proposition}
\newtheorem{definition}[theorem]{Definition}

\allowdisplaybreaks
\numberwithin{equation}{section}
\begin{document}
\title{Large time convergence for a chemotaxis model with degenerate local sensing and consumption} 

\author{Philippe Lauren\c{c}ot}
\address{Laboratoire de Math\'ematiques (LAMA) UMR~5127, Universit\'e Savoie Mont Blanc, CNRS\\
F--73000 Chamb\'ery, France}
\email{philippe.laurencot@univ-smb.fr}


\date{\today}

\begin{abstract}
Convergence to a steady state in the long term limit is established for global weak solutions to a chemotaxis model with degenerate local sensing and consumption, when the motility function is $C^1$-smooth on $[0,\infty)$, vanishes at zero, and is positive on $(0,\infty)$. A condition excluding that the large time limit is spatially homogeneous is also provided. These results extend previous ones derived for motility functions vanishing algebraically at zero and rely on a completely different approach.
\end{abstract}

\maketitle

%
%
\pagestyle{myheadings}
\markboth{\sc{Ph. Lauren\c cot}}{\sc{Large time convergence for a chemotaxis model}}

\section{Introduction}\label{sec1}

The chemotaxis system with local sensing and consumption 
\begin{subequations}\label{ks}
	\begin{align}
		\partial_t u & = \Delta (u\gamma(v)) \;\;\text{ in }\;\; (0,\infty)\times \Omega\,, \label{ks1} \\
		\partial_t v & = \Delta v - uv \;\;\text{ in }\;\; (0,\infty)\times \Omega\,, \label{ks2} \\
		\nabla (u\gamma(v))\cdot \mathbf{n} & = \nabla v\cdot \mathbf{n} = 0 \;\;\text{ on }\;\; (0,\infty)\times \partial\Omega\,, \label{ks3} \\
		(u,v)(0) & = (u^{in},v^{in}) \;\;\text{ in }\;\; \Omega\,, \label{ks4} 
	\end{align}
\end{subequations}
describes the dynamics of a population of cells with density $u\ge 0$ living on a nutrient with concentration $v\ge 0$ and moving in space under the combined effects of a nutrient-dependent diffusion and a nutrient-induced chemotactic bias \cite{KeSe1971b}. Here, $\Omega$ is a bounded domain of $\mathbb{R}^n$, $n\ge 1$, and the motility~$\gamma$ is a smooth function which is positive on $(0,\infty)$. Unlike the classical Keller-Segel system with local sensing
\begin{subequations}\label{cks}
	\begin{align}
		\partial_t u & = \Delta (u\gamma(v)) \;\;\text{ in }\;\; (0,\infty)\times \Omega\,, \label{cks1} \\
		\partial_t v & = \Delta v - v + u \;\;\text{ in }\;\; (0,\infty)\times \Omega\,, \label{cks2} \\
		\nabla (u\gamma(v))\cdot \mathbf{n} & = \nabla v\cdot \mathbf{n} = 0 \;\;\text{ on }\;\; (0,\infty)\times \partial\Omega\,, \label{cks3} \\
		(u,v)(0) & = (u^{in},v^{in}) \;\;\text{ in }\;\; \Omega\,, \label{cks4} 
	\end{align}
\end{subequations}
in which the variable $v$ is rather the concentration of a signalling chemical produced by the cells as accounted for in~\eqref{cks2} \cite{KeSe1970}, the nutrient is consumed by the cells in the model~\eqref{ks} according to the nonlinear absorption term $-uv$ in~\eqref{ks2}. The dynamics of the two models is thus expected to differ significantly. On the one hand, the long term behaviour of solutions to~\eqref{cks} is far from being completely understood. Convergence of $(u,v)$ to the spatially homogeneous steady state $(\|u^{in}\|_1/|\Omega|,\|u^{in}\|_1/|\Omega|)$ is established in \cite[Theorem~1.5]{DLTW2023} when $\gamma'(s)\le 0 \le s\gamma'(s) + \gamma(s)$ for $s\ge 0$ and a similar result is likely to be true when $\gamma'\ge 0$. Such a simple dynamics is however unlikely to be the generic one, as there exist non-constant stationary solutions to~\eqref{cks} in some domains, see~\cite{LNT1988, WaXu2021}. Besides, though solutions to~\eqref{cks} are global and even bounded for a large class of motility functions~$\gamma$ \cite{BPT2021, DKTY2019, DLTW2023, FuJi2021a, FuJi2021b, FuSe2022b, JLZ2022, LiJi2020, TaWi2017, XiJi2023, YoKi2017}, unbounded solutions do exist \cite{FuJi2022, FuSe2022a, JiWa2020}.

The situation is somewhat simpler for~\eqref{ks} as the set of stationary solutions can easily be identified and seen to depend heavily on the value of  $\gamma(0)$ \cite{LiWi2023a, Wink2023c}. Indeed, if $\gamma(0)>0$, then the only steady states to~\eqref{ks} are the spatially homogeneous solutions $(M,0)$, $M\ge 0$, and they attract the dynamics \cite{Laur2023, LiWi2023a}. In contrast, the set of stationary solutions to~\eqref{ks} is much larger when $\gamma(0)=0$, as $(\bar{u},0)$ is a stationary solution to~\eqref{ks} for any (sufficiently smooth) function~$\bar{u}$. Despite this wealth of stationary solutions, Winkler proves in~\cite{Wink2023b, Wink2023c} that the dynamics of~\eqref{ks} selects one and only one steady state in the large time limit: more precisely, given
\begin{equation}
	\gamma\in C^1([0,\infty))\,, \quad \gamma(0)=0, \quad \gamma>0 \;\;\text{ on }\;\; (0,\infty)\,, \label{n1}
\end{equation}
satisfying
\begin{equation}
	\gamma\in C^3((0,\infty)) \;\;\;\text{ and }\;\; \liminf_{s\to 0} \left\{ \frac{\gamma(s)}{s^\alpha} \right\} > 0 \label{w1}
\end{equation}
for some $\alpha\ge 1$ and a suitably constructed global weak solution~$(u,v)$ to~\eqref{ks} \cite{Wink2022a, Wink2023a}, there is a non-negative measurable function~$u_\infty$ such that
\begin{equation*}
	(u(t),v(t)) \;\;\text{ converges to }\; (u_\infty,0) \;\;\text{ in a suitable topology }
\end{equation*}
if one of the following additional assumptions holds true: 
\begin{description}
	\item[(I)] $n\in\{1,2\}$ and $\gamma'(0)>0$, see \cite[Theorem~1.4]{Wink2023b},
	\item[(II)] $n\ge 3$ and there is $\alpha'\in (1,2]$ such that
	\begin{equation}
		\limsup_{s\to 0} \left\{ s^{2-\alpha'} |\gamma''(s)| \right\} < \infty \,, \label{w2}
	\end{equation}
see \cite[Theorem~1.2]{Wink2023c}. 
\end{description}  
Observe that, for $\alpha\ge 1$, the function $\gamma(s)=s^\alpha$, $s\ge 0$, satisfies~\eqref{n1}, \eqref{w1}, and~\eqref{w2} with $\alpha'=\min\{\alpha, 2\}$. 

Besides, conditions on $\gamma$, $u^{in}$, and $v^{in}$ are provided to guarantee that $u_\infty$ is not a constant, see~\cite[Theorem~1.5]{Wink2023b} and~\cite[Corollary~1.4]{Wink2023c}.

The aim of this note is to prove that these results are actually valid under the sole assumption~\eqref{n1} on $\gamma$, without assuming an algebraic behaviour of $\gamma$ near zero. However, the convergence established below takes place in a weaker topology for the $u$-component than the one obtained in~\cite{Wink2023b, Wink2023c}. In addition, the approach used herein is different and thus provides an alternative viewpoint on the stabilization issue for~\eqref{ks}. We first state the convergence result.

\begin{theorem}\label{thm1}
	Assume that $\gamma$ satisfies~\eqref{n1} and consider $(u^{in},v^{in}) \in L_+^\infty(\Omega,\mathbb{R}^2)$, where $E_+$ denotes the positive cone of the Banach lattice $E$. If $(u,v)$ is a global weak solution to~\eqref{ks} in the sense of Definition~\ref{defws} below, then 
	\begin{equation*}
		A_\infty(x) := \int_0^\infty (u\gamma(v))(s,x)\ \mathrm{d}s \in H_+^1(\Omega)\,, \qquad x\in\Omega\,, 
	\end{equation*}
and
	\begin{align}
		& u(t) \rightharpoonup u_\infty := \Delta A_\infty + u^{in} \;\;\;\text{ in }\;\; H^1(\Omega)' \;\;\text{ as }\; t\to \infty\,, \label{cvu} \\ 
		& \lim_{t\to\infty} \|v(t)\|_p = 0\,, \qquad p\in [1,\infty)\,. \label{cvv} 
	\end{align}
Moreover, $\langle u_\infty , \vartheta \rangle_{(H^1)',H^1} \ge 0$ for all $\vartheta\in H_+^1(\Omega)$ and $\langle u_\infty , 1 \rangle_{(H^1)',H^1} = \langle u^{in}, 1\rangle_{(H^1)',H^1}$.
\end{theorem}

The cornerstone of our approach to the study of the long term behaviour of global weak solutions to~\eqref{ks}is the observation that the dynamics of the system~\eqref{ks} is somewhat encoded in that of the auxiliary function
\begin{equation*}
	A(t,x) := \int_0^t (u\gamma(v))(s,x)\ \mathrm{d}s\,, \qquad (t,x)\in (0,\infty)\times\Omega\,,
\end{equation*}
which bears several interesting properties collected in Lemma~\ref{lemb2}. Among others, $t\mapsto A(t,x)$ is non-decreasing for a.e. $x\in \Omega$ and the trajectory $\{A(t)\ :\ t\ge 0\}$ is bounded in $H^1(\Omega)$, two features which guarantee in particular that the function~$A_\infty$ introduced in Theorem~\ref{thm1} is well-defined and lies in $H^1(\Omega)$. Moreover, for all $t\ge 0$, the function~$A(t)$ is a variational solution of the elliptic equation 
\begin{equation*}
	- \Delta A(t) = u^{in} - u(t)\;\;\text{ in }\;\Omega\,, \qquad \nabla A(t)\cdot \mathbf{n} = 0 \;\;\text{ on }\;\partial\Omega\,,
\end{equation*}
so that the large time behaviour of $u(t)$ is driven by that of $A(t)$.

The second contribution of this paper is in the spirit of~\cite[Theorem~1.3]{Wink2023c} and provides an estimate on the distance in $H^1(\Omega)'$ between the initial condition $u^{in}$ and the final state $u_\infty$ of the $u$-component of~\eqref{ks}. Its statement requires additional notation, which we introduce now: for $z\in H^1(\Omega)'$, we set $\langle z\rangle := \langle z , 1 \rangle_{(H^1)',H^1}/|\Omega|$ and note that
\begin{equation*}
	\langle z\rangle = \frac{1}{|\Omega|} \int_\Omega z(x)\ \mathrm{d}x \;\;\;\text{ for }\;\; z\in H^1(\Omega)'\cap L^1(\Omega).
\end{equation*}
Now, for $z\in H^1(\Omega)'$ with $\langle z \rangle = 0$, we define $\mathcal{K}[z]\in H^1(\Omega)$ as the unique (variational) solution to 
\begin{subequations}\label{n2}
	\begin{equation}
		-\Delta\mathcal{K}[z] = z \;\;\text{ in }\;\; \Omega\,, \qquad \nabla\mathcal{K}[z]\cdot \mathbf{n} = 0 \;\;\text{ on }\;\; \partial\Omega\,, \label{n2a}
	\end{equation}
	satisfying
	\begin{equation}
		\langle\mathcal{K}[z] \rangle = 0\,. \label{n2b}
	\end{equation}
\end{subequations}
We then choose the following norm $\|\cdot\|_{(H^1)'}$ on $H^1(\Omega)'$:
\begin{equation*}
	\|z\|_{(H^1)'} := \|\nabla\mathcal{K}[z-\langle z\rangle]\|_2 + |\langle z\rangle|\,, \qquad z\in (H^1)(\Omega)'.
\end{equation*}

\begin{proposition}\label{prop2}
Assume that $\gamma$ satisfies~\eqref{n1} and consider $(u^{in},v^{in}) \in L_+^\infty(\Omega,\mathbb{R}^2)$. If $(u,v)$ is a global weak solution to~\eqref{ks} in the sense of Definition~\ref{defws} below	and $u_\infty$ denotes the weak limit in $H^1(\Omega)'$ of $u(t)$ as $t\to\infty$ given by Theorem~\ref{thm1}, then
	\begin{equation}
		\|u_\infty - u^{in}\|_{(H^1)'}^2 \le \|u^{in}\|_\infty \|v^{in}\|_1 \|\gamma'\|_{L^\infty(0,\|v^{in}\|_\infty)}\,. \label{dist}
	\end{equation}
In particular, if 
\begin{equation}
	\|u^{in}\|_\infty \|v^{in}\|_1 \|\gamma'\|_{L^\infty(0,\|v^{in}\|_\infty)} < \|u^{in} -  \langle u^{in} \rangle \|_{(H^1)'}^2\,, \label{smalldist}
\end{equation}
then $u_\infty$ is not a constant.
\end{proposition}

An immediate consequence of Proposition~\ref{prop2} is that, given $u^{in}\in L_+^\infty(\Omega)$ with $u^{in}\not\equiv \langle u^{in} \rangle$ and a sufficient small $v^{in}\in L_+^\infty(\Omega)$, the first component of the corresponding global weak solution $(u,v)$ to~\eqref{ks} has a non-constant limit. A quantitative estimate on the required smallness of $v^{in}$ in $L^\infty(\Omega)$ is provided by~\eqref{smalldist} and reads
\begin{equation*}
	 \|v^{in}\|_\infty \|\gamma'\|_{L^\infty(0,\|v^{in}\|_\infty)} < \frac{\|u^{in} -  \langle u^{in} \rangle \|_{(H^1)'}^2}{|\Omega|\|u^{in}\|_\infty}\,.
\end{equation*}
Such a result is obviously connected with the fact that the solution $(u,v)$ to~\eqref{ks} with initial condition $(u^{in},0)$ is the stationary solution $(u,v)=(u^{in},0)$, as already mentioned.

\section{Proofs}\label{sec2}

Let us first make precise the notion of global weak solution to~\eqref{ks} to be used in this paper. We emphasize here that, since $\gamma(0)=0$, the equation~\eqref{ks1} is degenerate, so that we cannot expect much regularity on~$u$.

\begin{definition}\label{defws}
	Assume that $\gamma$ satisfies~\eqref{n1} and consider $(u^{in},v^{in})\in L_+^\infty(\Omega,\mathbb{R}^2)$. A global weak solution to~\eqref{ks} is a pair of non-negative functions $(u,v)$ such that
	\begin{align*}
		u & \in C_w([0,\infty),H^1(\Omega)') \cap L^\infty((0,\infty),L_+^1(\Omega))\,, \\
		v & \in C([0,\infty),L_+^1(\Omega)) \cap L^\infty((0,\infty)\times\Omega)\cap L_{\mathrm{loc}}^2([0,\infty),H^1(\Omega))\,, \\
		u\sqrt{\gamma(v)} & \in L_{\mathrm{loc}}^2([0,\infty),L^2(\Omega))\,,
	\end{align*}
which satisfies
\begin{align*}
	\langle u(t) , \vartheta(t) \rangle_{(H^1)',H^1} - \int_\Omega u^{in}\vartheta(0)\ \mathrm{d}x & = \int_0^t \int_\Omega u(s)\gamma(v(s)) \Delta\vartheta(s)\ \mathrm{d}x\mathrm{d}s \\
	& \qquad + \int_0^t \langle u(s) , \partial_t \vartheta(s) \rangle_{(H^1)',H^1}\ \mathrm{d}s
\end{align*}
for $\vartheta\in L^2((0,t),H_N^2(\Omega))\cap W^{1,2}((0,t),H^1(\Omega))$ and $t\ge 0$, where
\begin{equation*}
	H_N^2(\Omega) := \{ z \in H^2(\Omega)\ :\ \nabla z\cdot \mathbf{n} = 0 \;\;\text{ on }\;\partial\Omega\}\,,
\end{equation*} 
as well as 
\begin{align*}
	\int_\Omega (v(t)\vartheta(t) - v^{in}\vartheta(0))\ \mathrm{d}x + \int_0^t \int_\Omega \nabla v(s)\cdot \nabla\vartheta(s)\ \mathrm{d}x\mathrm{d}s & + \int_0^t \int_\Omega (uv)(s)\vartheta(s)\ \mathrm{d}x\mathrm{d}s \\
	& = \int_0^t \int_\Omega v(s)\partial_t\vartheta(s)\ \mathrm{d}x\mathrm{d}s
\end{align*}
for $\vartheta\in L^2((0,t),H^1(\Omega))\cap L^\infty((0,t)\times\Omega)$ and $t\ge 0$.
\end{definition}

We shall not address the existence issue here and refer to \cite{Wink2022a, Wink2023a, Wink2023b} for results in that direction, the main assumption on $\gamma$ being that it vanishes in an algebraic way at zero. In the non-degenerate case $\gamma>0$ on $[0,\infty)$, implying in particular that $\gamma(0)>0$, existence results are also available, see \cite{LiZh2021, LiWi2023a, LiWi2023b}.

\bigskip

We now fix $\gamma$ satisfying~\eqref{n1} and consider $(u^{in},v^{in})\in L_+^\infty(\Omega,\mathbb{R}^2)$, along with a global weak solution $(u,v)$ to~\eqref{ks} in the sense of Definition~\ref{defws}. As a first step towards the identification of the large time limit of $(u,v)$, we collect obvious consequences of~\eqref{ks}, the non-negativity of $u$ and $v$, and the comparison principle.

\begin{lemma}\label{lemb1}
	For $t\ge 0$, 
	\begin{equation}
		\langle u(t) \rangle = M := \langle u^{in} \rangle \;\;\text{ and }\;\; \|v(t)\|_\infty \le V := \|v^{in}\|_\infty\,. \label{b1}
	\end{equation}
Moreover,
\begin{equation}
	\int_0^\infty \|(uv)(t)\|_1\ \mathrm{d}t \le \|v^{in}\|_1\,. \label{b2}
\end{equation}
\end{lemma}

\begin{proof}
	We integrate~\eqref{ks1} with respect to space and time and use the no-flux boundary conditions~\eqref{ks3} to obtain the first identity in~\eqref{b1}. Similarly, we infer from~\eqref{ks2} and~\eqref{ks3} that
	\begin{equation}
		\frac{\mathrm{d}}{\mathrm{d}t} \|v(t)\|_1 + \int_0^t \|(uv)(s)\|_1\ \mathrm{d}s = \|v^{in}\|_1\,, \qquad t\ge 0\,, \label{b3}
	\end{equation}
from which~\eqref{b2} readily follows. Finally, we use the comparison principle to deduce from~\eqref{ks2}, \eqref{ks3}, and the non-negativity of $uv$ that $v(t,x)\le V$ for $(t,x)\in [0,\infty)\times\bar{\Omega}$, thereby completing the proof of~\eqref{b1}.
\end{proof}

We now define the auxiliary function
\begin{equation*}
	A(t,x) := \int_0^t (u\gamma(v))(s,x)\ \mathrm{d}s\,, \qquad (t,x)\in (0,\infty)\times \Omega\,, 
\end{equation*} 
and devote the next lemma to its properties. 

\begin{lemma}\label{lemb2}
	The function~$A$ belongs to $L^\infty((0,\infty),H^1(\Omega))$ with
	\begin{subequations}\label{b45}
	\begin{align}
		\|A(t)\|_1 & \le \|v^{in}\|_1 \|\gamma'\|_{L^\infty(0,V)}\,, \qquad t\ge 0 \,, \label{b4} \\
		\|\nabla A(t)\|_2^2 & \le \|u^{in}\|_\infty \|v^{in}\|_1 \|\gamma'\|_{L^\infty(0,V)}\,, \qquad t\ge 0\,.\label{b5}
	\end{align}
\end{subequations}
In addition, $t\mapsto A(t,x)$ is a non-decreasing function for a.e. $x\in \Omega$ and
\begin{equation}
	A_\infty(x) := \sup_{t\ge 0}\{A(t,x)\} = \int_0^\infty (u\gamma(v))(s,x)\ \mathrm{d}s\,, \qquad x\in\Omega\,, \label{b6}
\end{equation}
is well-defined and belongs to $H_+^1(\Omega)$. Also, for any $\vartheta\in H^1(\Omega)$,
\begin{equation}
	\lim_{t\to\infty} \|A(t)-A_\infty\|_2 = \lim_{t\to\infty} \int_\Omega \nabla\vartheta\cdot \nabla(A(t)-A_\infty)\ \mathrm{d}x = 0\,. \label{b7}
\end{equation}
\end{lemma}

\begin{proof}
	Owing to the non-negativity of $u$ and $\gamma$, 
	\begin{equation*}
		A(t_1,x) \le A(t_2,x)\,, \qquad 0\le t_1 \le t_2\,, \ x\in\Omega\,, 
	\end{equation*}
while, for $t\ge 0$, it follows from~\eqref{n1}, \eqref{b1}, and~\eqref{b2} that
\begin{align*}
	\|A(t)\|_1 & = \int_0^t \int_\Omega (u\gamma(v))(s,x)\ \mathrm{d}x\mathrm{d}s \le \|\gamma'\|_{L^\infty(0,V)} \int_0^t \int_\Omega (uv)(s,x)\ \mathrm{d}x\mathrm{d}s \\
	& \le \|v^{in}\|_1 \|\gamma'\|_{L^\infty(0,V)} \,,
\end{align*}
which proves~\eqref{b4}. Furthermore, the monotone convergence theorem implies that the function $A_\infty$ defined by~\eqref{b6} belongs to $L_+^1(\Omega)$ and 
\begin{equation}
	\lim_{t\to\infty} \|A(t)-A_\infty\|_1 = 0\,. \label{b8}
\end{equation}

We next infer from~\eqref{ks1}, \eqref{ks3}, and the definition of $A$ that, for $t\ge 0$, 
\begin{equation}
	u(t) - \Delta A(t) = u^{in} \;\;\;\text{ in }\;\; H^1(\Omega)'\,. \label{b9}
\end{equation}
In particular, $A(t) - \langle A(t)\rangle = \mathcal{K}[u^{in}-u(t)]$ belongs to $H^1(\Omega)$ and we infer from~\eqref{b9} and the non-negativity of $u(t)$ and $A(t)$ that
\begin{align*}
	\|\nabla A(t)\|_2^2 & = \langle - \Delta A(t), A(t) \rangle_{(H^1)',H^1} \\
	& \le \langle u(t) - \Delta A(t) , A(t) \rangle_{(H^1)',H^1} = \int_\Omega u^{in} A(t)\ \mathrm{d}x \\
	& \le \|u^{in}\|_\infty \|A(t)\|_1\,.
\end{align*}
Hence, by~\eqref{b4},
\begin{equation*}
	\|\nabla A(t)\|_2^2 \le \|u^{in}\|_\infty \|v^{in}\|_1 \|\gamma'\|_{L^\infty(0,V)}\,,
\end{equation*}
from which~\eqref{b5} follows. Finally, we deduce the $H^1$-regularity of $A_\infty$ from~\eqref{b5} by a weak compactness argument, whereas the convergence~\eqref{b7} is an immediate consequence of~\eqref{b4}, \eqref{b5}, and~\eqref{b8}.
\end{proof}

We next turn to the convergence of~$v$ and begin with a classical energy estimate, which is available here thanks to the non-negativity of the right hand side  of~\eqref{ks2}.

\begin{lemma}\label{lemb3}
	For $t\ge 0$,
	\begin{equation*}
		\frac{\mathrm{d}}{\mathrm{d}t} \|v\|_2^2 + 2 \|\nabla v\|_2^2 + 2 \|v\sqrt{u}\|_2^2 = 0\,. 
	\end{equation*}
\end{lemma}

\begin{lemma}\label{lemb4}
	For each $p\in [1,\infty)$,
	\begin{equation*}
		\lim_{t\to\infty} \|v(t)\|_p = 0\,. 
	\end{equation*}
\end{lemma}

\begin{proof}
	Introducing $P:=\mathcal{K}[u-M]$ and $P^{in} := \mathcal{K}[u^{in}-M]$, we observe that $P = P^{in}-A$, so that
	\begin{equation}
		\|\nabla P\|_2 \le \|\nabla P^{in}\|_2 + \|\nabla A\|_2 \le c_1 := \|u^{in}\|_2 + \sqrt{\|u^{in}\|_\infty \|v^{in}\|_1 \|\gamma'\|_{L^\infty(0,V)}}\,. \label{b10}
	\end{equation}
	We next infer from~\eqref{n2a} and~\eqref{b3} that
	\begin{equation*}
		\frac{\mathrm{d}}{\mathrm{d}t} \|v\|_1 = - \int_\Omega uv\ \mathrm{d}x = - \int_\Omega (M-\Delta P) v\ \mathrm{d}x = - M \|v\|_1  + \int_\Omega \nabla v\cdot \nabla P\ \mathrm{d}x\,.
	\end{equation*}
	Hence, using~\eqref{b10} and H\"older's inequality,
	\begin{equation*}
		\frac{\mathrm{d}}{\mathrm{d}t} \|v\|_1 + M \|v\|_1 \le \|\nabla v\|_2 \|\nabla P\|_2 \le c_1 \|\nabla v\|_2\,.
	\end{equation*}
	After integration with respect to time, we obtain
	\begin{equation}
		\|v(t)\|_1 \le \|v^{in}\|_1 e^{-Mt} + c_1 \int_0^t e^{M(s-t)} \|\nabla v(s)\|_2\ \mathrm{d}s\,, \qquad t\ge 0\,. \label{b11}
	\end{equation}
	Now, by H\"older's inequality,
	\begin{align*}
		\int_0^t e^{M(s-t)} \|\nabla v(s)\|_2\ \mathrm{d}s & \le \left( \int_0^t e^{M(s-t)}\ \mathrm{d}s \right)^{1/2} \left( \int_0^t e^{M(s-t)} \|\nabla v(s)\|_2^2\ \mathrm{d}s \right)^{1/2} \\
		& \le \frac{1}{\sqrt{M}} \left( \int_0^t e^{M(s-t)} \|\nabla v(s)\|_2^2\ \mathrm{d}s \right)^{1/2}\,.
	\end{align*}
	Since
	\begin{equation*}
		2 \int_0^\infty \|\nabla v(s)\|_2^2\ \mathrm{d}s \le \|v^{in}\|_2^2
	\end{equation*}
by Lemma~\ref{lemb3}, we deduce from the Lebesgue dominated convergence theorem that 
	\begin{equation*}
		\lim_{t\to\infty} \int_0^t e^{M(s-t)} \|\nabla v(s)\|_2^2\ \mathrm{d}s = 0\,.
	\end{equation*}
	Consequently,
	\begin{equation}
		\int_0^t e^{M(s-t)} \|\nabla v(s)\|_2\ \mathrm{d}s = 0 \label{b12}
	\end{equation}
	and~\eqref{b11} and~\eqref{b12} entail that
	\begin{equation*}
		\lim_{t\to\infty} \|v(t)\|_1 = 0\,,
	\end{equation*}
	thereby proving Lemma~\ref{lemb4} for $p=1$. To complete the proof, we use the above convergence, along with~\eqref{b1} and H\"older's inequality.
\end{proof}

Thanks to the above analysis, we are now in a position to prove Theorem~\ref{thm1} and Proposition~\ref{prop2}.

\begin{proof}[Proof of Theorem~\ref{thm1}]
	According to Lemma~\ref{lemb2}, the function $A_\infty$ introduced in Theorem~\ref{thm1} is well-defined and belongs to $H_+^1(\Omega)$. Setting $u_\infty = u^{in} + \Delta A_\infty\in H^1(\Omega)'$, we infer from~\eqref{b9} that, for $t\ge 0$ and $\vartheta\in H^1(\Omega)$,
	\begin{align*}
		\langle u(t)-u_\infty , \vartheta \rangle_{(H^1)',H^1} & = \langle u(t)-u^{in} + u^{in}-u_\infty , \vartheta \rangle_{(H^1)',H^1} \\
		& = \langle \Delta (A(t) - A_\infty) , \vartheta \rangle_{(H^1)',H^1} \\
		& = - \int_\Omega \nabla(A(t)-A_\infty)\cdot \nabla\vartheta\ \mathrm{d}x\,,
	\end{align*}  
and the right hand side of the above identity converges to zero as $t\to\infty$ due to~\eqref{b7}. We have thus proved the convergence~\eqref{cvu}, whereas the convergence~\eqref{cvv} is established in Lemma~\ref{lemb4}. As for the properties of~$u_\infty$ stated at the end of Theorem~\ref{thm1}, they readily follow from~\eqref{b1}, the non-negativity, and the convergence~\eqref{cvu}
\end{proof}

\begin{proof}[Proof of Proposition~\ref{prop2}]
	The starting point is the estimate~\eqref{b5} and the convergences~\eqref{b7} which imply that
	\begin{equation}
		\|\nabla A_\infty\|_2^2 \le \|u^{in}\|_\infty \|v^{in}\|_1 \|\gamma'\|_{L^\infty(0,V)}\,, \label{b13}
	\end{equation}
recalling that $V=\|v^{in}\|_\infty$. Since $\langle u_\infty \rangle = \langle u^{in} \rangle$ by Theorem~\ref{thm1}, it follows from~\eqref{b13} and the definition of $u_\infty$ that
\begin{equation*}
	\|u_\infty - u^{in}\|_{(H^1)'}^2 = \|\nabla\mathcal{K}[u_\infty - u^{in}]\|_2^2 = \|\nabla A_\infty\|_2^2 \le \|u^{in}\|_\infty \|v^{in}\|_1 \|\gamma'\|_{L^\infty(0,V)}\,,
\end{equation*}
as stated in~\eqref{dist}. 

Assume now that $(u^{in},v^{in})$ satisfies~\eqref{smalldist}. It follows from~\eqref{dist} and~\eqref{smalldist} that
\begin{align*}
	\|u_\infty - \langle u^{in} \rangle \|_{(H^1)'} & = \|u_\infty - u^{in} + u^{in} - \langle u^{in} \rangle \|_{(H^1)'} \\
	& \ge \|u^{in} - \langle u^{in} \rangle \|_{(H^1)'} - \|u_\infty - u^{in} \rangle \|_{(H^1)'} \\
	& > \left[ \|u^{in}\|_\infty \|v^{in}\|_1 \|\gamma'\|_{L^\infty(0,V)} \right]^{1/2} - \|u_\infty - u^{in} \rangle \|_{(H^1)'} >0\,.
\end{align*}
Consequently, $u_\infty \ne \langle u^{in} \rangle$, which completes the proof after noticing that the property $\langle u_\infty \rangle =  \langle u^{in} \rangle$ established in Theorem~\ref{thm1} excludes that $u_\infty$ coincides with any other constant. 
\end{proof}

\section*{Acknowledgments}

Enlightening (electronic) discussions with Michael Winkler on the topic studied in this paper are gratefully acknowledged. Part of this work was done while enjoying the kind hospitality of the Department of Mathematics, Indian Institute of Technology Roorkee.
 
\bibliographystyle{siam}
\bibliography{LTCCMDLSC}

\end{document}